\documentclass[a4]{article}
\usepackage{amsmath,amssymb,amsthm,amsopn, framed, mathrsfs,url}

\theoremstyle{plain}
\newtheorem{thm}{Theorem}[section]
\newtheorem{lem}[thm]{Lemma}

\theoremstyle{definition}
\newtheorem*{proof1}{Proof of Theorem \ref{main}}

\newmuskip\pFqmuskip

\newcommand*\pFq[6][8]{%
  \begingroup 
  \pFqmuskip=#1mu\relax
  \mathcode`\,=\string"8000
  \begingroup\lccode`\~=`\,
  \lowercase{\endgroup\let~}\pFqcomma
  {}_{#2}F_{#3}{\left(\genfrac..{0pt}{}{#4}{#5};#6\right)}%
  \endgroup
}
\newcommand{\pFqcomma}{\mskip\pFqmuskip}

\begin{document}

\title{A Picard family of curves and hypergeometric functions over finite fields I}
\author{Yoh Takizawa}
\date{}
\maketitle

\begin{abstract}
We give an expression for the trace of Frobenius for the family of curves
\[ y^3 = x (x-1)(x-\lambda)(x-\mu) \]
over finite fields in terms of finite field hypergeometric functions.
\end{abstract}

\section{Introduction}
Let $p$ be an odd prime and $\mathbb{F}_p$ be the finite fields with $p$ elements.
For $\lambda \in \mathbb{F}_q \setminus \{ 0, 1 \}$ we define an elliptic curve over $\mathbb{F}_q$ 
in the Legendre family by 
\[ E : y^2 = x(x-1)(x- \lambda) \]
Koike \cite{Koike} showed that, for all odd primes $p$, 
the trace of Frobenius for curves in this family can be expressed in terms of 
Grreene's finite field hypergeometric function. 
\[ a_p(E) = -p \phi(-1) \pFq{2}{1}{\phi, \phi}{\varepsilon}{\lambda}_p \]
where $\varepsilon$ is the trivial character and $\phi$ is a quadratic character of $\mathbb{F}_p^{\times}$. 

\[\]

Let $q$ be a power of a rational prime $p$ and $\mathbb{F}_q$ be the finite fields with $q$ elements.
For $\lambda, \mu \in \mathbb{F}_q$ with $\lambda \mu (\lambda - 1)(\mu - 1) \neq 0$ 
we define a smooth projective curve over $\mathbb{F}_q$ in the Picard family by 
\[ C : y^3 = x(x-1)(x- \lambda)(x-\mu). \]
This is a genus $3$ curve.

We show that the trace of Frobenius $a_q(C)$ for $C$ can be expressed in terms of finite field Appell hypergeometric functions. 

\begin{thm}\label{main}
Let $q = p^e$ be a power of prime such that $q \equiv 1 \pmod 3$, 
 $\chi_3$ be a cubic character and  $\varepsilon$ be the trivial character of $\mathbb{F}_q$, we have
\[ a_q(C) = -q \sum^2_{J=1} \chi^{2j}_3(-1) \pFq{}{1}{\chi^j_3, \chi^j_3, \chi^j_3}{\varepsilon}{\lambda, \mu}_q. \]
\end{thm}

\section{Finite field hypergeometric functions}

Recall that the hyperheometric series is defined by 
\[ \pFq{2}{1}{a, b}{c}{x} = \sum^{\infty}_{n=0} \dfrac{(a, n)(b, n)}{(1, n)(c, n)} x^n \]
where
\[ (a, n) = \begin{cases}
a(a+1) \cdots (a+n-1), \ \ \ \ \ &n \ge 1\\
1, &n = 0
\end{cases} \]

In \cite{Greene}, Greene defined a finite field analogue of classical hypergeometric series. 
Let $q = p^a$ a power of prime, $\mathbb{F}_q$ be a finite field of $q$ elements.
For a character $\chi \in \widehat{\mathbb{F}_q^{\times}}$, we extend it to all of $\mathbb{F}_q$ by setting
\[ \chi(0) = \begin{cases}
0, \ \ \ \ \ \ &\chi \neq \varepsilon\\
1, &\chi = \varepsilon.
\end{cases} \]
where $\varepsilon$ is the trivial character.

For two characters $A, B$ of $\mathbb{F}_q$ we define the normalized Jacobi sum 
\[{A \choose B} = \dfrac{B(-1)}{q} J(A, \overline{B}) \]
where $J(A, B) = \sum_{t \in \mathbb{F}_q} A(t) B(1-t)$ is the usual Jacobi sum. 
The following properties of normalized Jacobi sum are found in \cite{Greene}.
\begin{lem}\label{lem:21}
Let $A, B$ be characters of $\mathbb{F}_q$ we have
\begin{enumerate}
\item[(1)] 
\[ \overline{A}(1-t) = \delta(t) + \dfrac{q}{q-1} \sum_{\chi \in \widehat{k}} {A \chi \choose \chi} \chi(t) \]
\item[(2)] 
\[ J(A, B) = q B(-1) {A \choose \overline{B}} \]
\item[(3)] 
\[ {A \choose B} = {A \choose A \overline{B}} \]
\end{enumerate}
where
\[ \delta(t) = \begin{cases}
0, \ \ \ \ \ &t \neq 0\\
1, &t = 0
\end{cases} \]
for $t \in \mathbb{F}_q$. 
\end{lem}

For three characters $A, B, C \in \widehat{\mathbb{F}_q^{\times}}$ Greene defined the finite field 
hypergeometric function ${}_2F_1$ by
\[ \pFq{2}{1}{A, B}{C}{x}_q = \dfrac{q}{q-1} 
\sum_{\chi \in \widehat{\mathbb{F}_q^{\times}}}{A \chi \choose \chi} {B \chi \choose C \chi} \chi(x)  \]

For four characters $A, B_1, B_2, C \in \widehat{\mathbb{F}_q^{\times}}$ Ghosh defined the finite field 
hypergeometric function of two variables $F_1$ by
\[ \pFq{}{1}{A, B_1, B_2}{C}{x, y}_q = \dfrac{q^2}{(q-1)^2} 
\sum_{\chi_1, \chi_2 \in \widehat{\mathbb{F}_q^{\times}}}{A \chi \choose C \chi} 
{B_1 \chi_1 \choose \chi_1} {B_2 \chi_2 \choose \chi_2} \chi_1(x) \chi_2(y)  \]
where $\chi = \chi_1 \chi_2$. 

This function is a finite field analogue of Appell hypergeometric series $F_1$. 
\[ \pFq{}{1}{a, b_1, b_2}{c}{x, y} = \sum_{m, n \ge 0} \dfrac{(a,m+n)(b_1,m)(b_2,n)}{(1,m)(1,n)(c,m+n)} x^m y^n. \]

The next lemma is proved by Ghosh \cite{Ghosh}. 

\begin{lem}\label{lem:22}iGhoshj
\[ \pFq{}{1}{A, B_1, B_2}{C}{x, y}_q =  \varepsilon(xy) \dfrac{AC(-1)}{q} 
\sum_{t \in \mathbb{F}_q} A(t) \overline{A} C(1-t) \overline{B}_1 (1-xt) \overline{B}_2 (1-yt) \]
where $\chi = \chi_1 \chi_2$. 
\end{lem}
\begin{proof}
Supoose 
\[ g(x, y) = \varepsilon(xy) \dfrac{AC(-1)}{p} 
\sum_{t \in \mathbb{F}_q} A(t) \overline{A} C(1-t) \overline{B}_1 (1-xt) \overline{B}_2 (1-yt) \]
and we consider the cases in $xy \neq 0$. 
\[ \begin{split}
\overline{B_1}(1-xt) \overline{B_2}(1-yt) &= 
\dfrac{q^2}{(q-1)^2} \Bigl( \sum_{\chi_1 \in \widehat{\mathbb{F}_q^{\times}}} {B_1 \chi_1 \choose \chi_1} \chi_1(xt) \Bigr) 
\Bigl( \sum_{\chi_2 \in \widehat{\mathbb{F}_q^{\times}}} {B_2 \chi_2 \choose \chi_2} \chi_2(yt) \Bigr)\\
&= \dfrac{q^2}{(q-1)^2} \sum_{\chi_1, \chi_2 \in \widehat{\mathbb{F}_q^{\times}}} {B_1 \chi_1 \choose \chi_1}  
{B_2 \chi_2 \choose \chi_2} \chi_1(xt) \chi_2(yt)\\
&= \dfrac{q^2}{(q-1)^2} \sum_{\chi_1, \chi_2 \in \widehat{\mathbb{F}_q^{\times}}} {B_1 \chi_1 \choose \chi_1}  
{B_2 \chi_2 \choose \chi_2} \chi(t) \chi_1(x) \chi_2(y)
\end{split}. \]
By (\ref{lem:21}) we have
\[ \begin{split}
g(x, y) &= \dfrac{q^2 AC(-1)}{(q-1)^2} 
\sum_{t \in \mathbb{F}_q} A(t) \overline{A} C(1-t) \sum_{\chi_1, \chi_2 \in \widehat{k}} {B_1 \chi_1 \choose \chi_1}  
{B_2 \chi_2 \choose \chi_2} \chi(t) \chi_1(x) \chi_2(y)\\
&= \dfrac{q^2 AC(-1)}{(q-1)^2} \sum_{\chi_1, \chi_2 \in \widehat{\mathbb{F}_q^{\times}}} {B_1 \chi_1 \choose \chi_1}  
{B_2 \chi_2 \choose \chi_2} \chi_1(x) \chi_2(y) \sum_{t \in \mathbb{F}_q} A \chi(t) \overline{A} C(1-t)\\
&= \dfrac{q^2 AC(-1)}{(q-1)^2} \sum_{\chi_1, \chi_2 \in \widehat{\mathbb{F}_q^{\times}}} {B_1 \chi_1 \choose \chi_1}  
{B_2 \chi_2 \choose \chi_2} \chi_1(x) \chi_2(y) J(A \chi, \overline{A} C)\\
&= \dfrac{q^2 AC(-1)}{(q-1)^2} \sum_{\chi_1, \chi_2 \in \widehat{\mathbb{F}_q^{\times}}} {B_1 \chi_1 \choose \chi_1}  
{B_2 \chi_2 \choose \chi_2} \chi_1(x) \chi_2(y) q \overline{A} C(-1) {A \chi \choose \overline{A} C}\\
&= \dfrac{q^2}{(q-1)^2} \sum_{\chi_1, \chi_2 \in \widehat{\mathbb{F}_q^{\times}}} {A \chi \choose \overline{A} C} 
{B_1 \chi_1 \choose \chi_1}  {B_2 \chi_2 \choose \chi_2} \chi_1(x) \chi_2(y)
\end{split}. \]
\end{proof}

\section{Picard curves over finite fields}

Let $q = p^a$ a power of prime $p > 3$, $\mathbb{F}_q$ be a finite field of $q$ elements 
and let $C$ be a smooth projective curve of genus $3$ over $\mathbb{F}_q$ with an affine model
\[ y^3 = x^4 + c_3x^3 + c_2x^2 + c_1x + c_0. \]
here it is supposed $f(x) = x^4 + c_3x^3 + c_2x^2 + c_1x + c_0$ has no multiple root. 
It is called a Picard curve. 

For $\lambda, \mu \in \mathbb{F}_q$ with $\lambda \mu (\lambda - 1)(\mu - 1) \neq 0$ 
we define a Picard curve over $\mathbb{F}_q$ by 
\[ C : y^3 = x(x-1)(x- \lambda)(x-\mu) \]
This is the Picard family of curves. 

Let $\sharp C(\mathbb{F}_q)$ be the number of $\mathbb{F}_q$-rational points of $C$ and we let 
\[ a_q(C) = 1 + q - \sharp C(\mathbb{F}_q). \]
 $a_q(C)$ is called the trace of Frobenius for $C$. 

The number of $\mathbb{F}_q$-rational points of $C$ can be expressed in 
terms of characters of $\mathbb{F}_q^{\times}$. 
\[ \sharp C(\mathbb{F}_q) = 1 + \sum_{\chi^3 = \varepsilon} \chi(x(x-1)(x-\lambda)(x-\mu)) \]
Let $\chi_3$ be a cubic character of $\mathbb{F}_q^{\times}$ we have
\[ \begin{split}
\sharp C(\mathbb{F}_q) &= 1 + \sum_{x \in \mathbb{F}_q} \sum^2_{j=0} \chi^j_3(x(x-1)(x-\lambda)(x-\mu))\\
&= 1 + q + \sum_{x \in \mathbb{F}_q} \sum^2_{j=1} \chi^j_3(x(x-1)(x-\lambda)(x-\mu))
\end{split} \]

\section{Proof of main theorem}
\begin{lem}\label{lem:23}
Let $A, B_1, B_2, C$ be characters of $\mathbb{F}_q$ and $x, y \in \mathbb{F}_q$ we have
\[ \pFq{}{1}{A, B_1, B_2}{C}{x, y}_q = \varepsilon(xy) \dfrac{AC(-1)}{p} 
\sum  B_1 B_2 \overline{C}(s) 
\overline{A} C(s-1) \overline{B}_1 (s-x) \overline{B}_2 (s-y) \]
\end{lem}
\begin{proof}
By (\ref{lem:22}) 
\[ \pFq{}{1}{A, B_1, B_2}{C}{x, y}_q =  \varepsilon(xy) \dfrac{AC(-1)}{q}  
\sum_t A(t) \overline{A} C(1-t) \overline{B}_1 (1-xt) \overline{B}_2 (1-yt). \]
Suppose $xy \neq 0$ and set $s = t^{-1}$,   
\[ \begin{split}
\pFq{}{1}{A, B_1, B_2}{C}{x, y}_q &= \dfrac{AC(-1)}{q}
\sum A(t) \overline{A} C(1-t) \overline{B}_1 (1-xt) \overline{B}_2 (1-yt)\\
&= \dfrac{AC(-1)}{q} 
\sum A \Bigl( \dfrac{1}{s} \Bigr) \overline{A} C \Bigl( \dfrac{s-1}{s} \Bigr) 
\overline{B}_1  \Bigl( \dfrac{s-x}{s} \Bigr) \overline{B}_2  \Bigl( \dfrac{s-y}{s} \Bigr)\\
&= \dfrac{AC(-1)}{q}
\sum B_1 B_2 \overline{C}(s) \overline{A} C(s-1) \overline{B}_1 (s-x) \overline{B}_2 (s-y)
\end{split} \]
\end{proof}


\begin{proof1}
The trace of Frobenius for $C$ is given by
\[ \begin{split}
a_q(C) &= 1 + q - \sharp C(\mathbb{F}_q)\\
&= - \sum_{x \in \mathbb{F}_q} \chi_3(x(x-1)(x-\lambda)(x-\mu))\\
&\ \ \ \ \ \ \ \  - \sum_{x \in \mathbb{F}_q} \chi_3^2(x(x-1)(x-\lambda)(x-\mu)).
\end{split} \]

Let $A = B_1 = B_2 = \chi_3$, $C = \varepsilon$ the hypergeometric function has the form 
\[ \begin{split}
\pFq{}{1}{\chi_3, \chi_3, \chi_3}{\varepsilon}{x, y}_q &= \dfrac{\chi_3(-1)}{q} 
\sum \chi_3^2(s) \overline{\chi_3} (s-1) \overline{\chi_3} (s-x) \overline{\chi_3} (s-y)\\
&= \dfrac{\chi_3(-1)}{q} 
\sum \overline{\chi_3}(s(s-1)(s-x)(s-y)).  
\end{split} \]
We get
\[ \sum_{t \in \mathbb{F}_q} \chi^2_3(t(t-1)(t-x)(t-y)) = q \chi^2_3(-1) 
\pFq{}{1}{\chi_3, \chi_3, \chi_3}{\varepsilon}{x, y}_q.  \]

Similarly, let $A = B_1 = B_2 = \chi_3^2$, $C = \varepsilon$ 
\[ \sum_{t \in \mathbb{F}_q} \chi^2_3(t(t-1)(t-x)(t-y)) = q \chi^2_3(-1) 
\pFq{}{1}{\chi_3, \chi_3, \chi_3}{\varepsilon}{x, y}_q.  \]
\qed
\end{proof1}

\end{document}